\documentclass{article}
\usepackage[T2A]{fontenc}
\usepackage[cp1251]{inputenc} 
\usepackage[tbtags]{amsmath}
\usepackage{amsfonts,amssymb,mathrsfs,amscd,comment, dsfont,amsthm,color}

\def\cL{\mathcal{L}}

\def\cL{\mathcal{L}}

\DeclareMathOperator{\supp}{supp}

\renewcommand{\ge}{\geqslant}
\renewcommand{\le}{\leqslant}

\numberwithin{equation}{section}

\newtheorem{theorem}{Theorem}
\newtheorem*{corollary}{Corollary}
\newtheorem*{remark}{Remark}

\begin{document}

\title{\bf\large\MakeUppercase{
Non-local convolution type operators with potential: essential and infinite discrete spectrum
}}

\author{D.I. Borisov$^{1,2}$, A.L. Piatnitskii$^{3,4}$, E.A. Zhizhina$^{4,3}$}
\date{}

\maketitle
{\small
    \begin{quote}
1) Institute of Mathematics, Ufa Federal Research Center, Russian Academy of Sciences,  Chernyshevsky str. 112, Ufa, Russia, 450008
\\
2) University of Hradec Kr\'alov\'e
62, Rokitansk\'eho, Hradec Kr\'alov\'e 50003, Czech Republic
\\
3) The Arctic University of Norway, campus Narvik, PO Box 385, Narvik 8505, Norway
\\
4) Institute for Information Transmission Problem (Kharkevich Institute), Bolshoy Karetny per. 19, build.1, Moscow, Russia, 127051
\\[2mm]
Emails: borisovdi@yandex.ru,  apiatnitski@gmail.com, ejj@iitp.ru
\end{quote}}




\begin{abstract}
  The goal of this note is to study the spectrum of a self-adjoint convolution operator
  in $L^2(\mathbb R^d)$ with an integrable kernel that is perturbed by
  an essentially bounded real-valued potential tending to zero at infinity.
We show that the essential spectrum of such operator is the union of the spectrum of
the convolution operator and of the essential range of the potential.
Then we provide several sufficient conditions for the existence of a countable sequence
of discrete eigenvalues.
For operators having non-connected essential spectrum we give sufficient conditions for the existence of discrete
eigenvalues in the corresponding spectral gaps.
\end{abstract}

\section{Introduction}\label{s_intr}
In this paper
we study the spectrum of self-adjoint non-local convolution type operators with a potential of the form
\begin{equation}\label{L}
(\cL u)(x):=\int\limits_{\mathds{R}^d} a(x-y) u(y)\,dy + V(x)u(x)\quad \text{in}\quad L_2(\mathds{R}^d).
\end{equation}
Our goal is to determine the location of the essential spectrum of operator $\cL$ and to provide conditions on the functions $a$ and $V$  ensuring the existence of countably many
points of the discrete spectrum.

In our previous work \cite{BPZ-JMAA} we studied  a similar spectral problem 
under the assumption that $V$ is the Fourier image of some function $\hat{V} \in L_1(\mathds{R}^d)$.
Here we drop this conditions on the potential $V$ and only assume that $V$ is a $L_\infty$ function that tends to zero at infinity. For such a potential the essential spectrum of  $\cL$ is getting more complicated, in particular, spectral gaps can appear. Our first result describes explicitly the essential component of the spectrum. 

Then we provide simple sufficient conditions guaranteeing the existence of infinitely many discrete eigenvalues below or above the essential spectrum.
We stress that  in \cite{BPZ-JMAA} the conditions ensuring such a structure of the discrete spectrum have been formulated in terms of the behaviour of higher
order Taylor or  Fourier coefficients of $a(\cdot)$ and $V$ localized in the vicinity of  their extreme points. 
This imposed in particular quite restrictive regularity assumptions on the functions $a(\cdot)$ and $V$.  Moreover, 
these conditions are rather implicit, it is difficult to check if they hold true.
In the present work sufficient conditions for the existence of infinitely many points of the discrete spectrum are given
in terms of agreed lower bounds on $V$ at infinity and on the Fourier image of $ a(\cdot)$ in the vicinity of its maximum point, no regularity of
these functions is required.

We also show that there can be additional eigenvalues in the spectral gaps and provide sufficient conditions ensuring the existence of such eigenvalues.

In recent years, there has been an increasing interest of mathematicians in non-local operators of convolutional type with an integrable kernel.
Such interest is motivated by the fact that such operators possess many interesting nontrivial  properties, which are not exhibited by classical differential operators. One more reason is
the presence of non-trivial qualitative and asymptotic problems
in this theory, for instance, 
various non-local homogenization problems, local estimates for the fundamental solution
and the Green function, large time asymptotics of the fundamental solutions, spectral problems, etc.

In applications, non-local convolution type  operators appear in such fields as population dynamics, porous media, image processing, see \cite{Rossi_book}, \cite{FKK} and \cite{KPZ} for further details. One of the mathematical tools widely used in  population dynamics is the so-called contact processes in continuum, see e.g.
\cite{KKP}, \cite{PZ}.   These processes are a particular case of  continuous time birth and death processes with infinite particle configurations in continuum. The function $a(\cdot)$ is called the dispersal kernel and it defines the distribution of a position of a newly born particle in the configuration. The mortality rate determines the intensity of  death,  in heterogeneous environments  it depends on a position in the space. This leads to the appearance of a non-constant potential $V(\cdot)$. One way of describing the evolution of stochastic infinite-particle configurations in continuum is based on studying a hierarchical system of evolution equations for the corresponding correlation functions. The equation for the first correlation function is decoupled and it reads as
$$
\partial_t u(x,t)=\cL u(x,t)\quad \text{in}\quad \mathbb R^d\times(0,+\infty),\qquad u(x,0)=u_0(x).
$$
Since the first correlation function represents the density of population, the large time behaviour of the population is characterized  by the spectrum of $\cL$. In particular, once the operator  $\cL$ has points of discrete spectrum above the top of the essential spectrum, the population shows an exponential growth, see \cite{KPMZh} and \cite{KMV} for further
discussion on this subject.

In some applications the description of various processes based on non-local convolution type operators is more accurate
than the description based on  differential equations. The latter provides an approximation that is suitable
for characterizing the macroscopic behaviour of the studied models and for obtaining the large time asymptotics of the corresponding evolution processes. This is in a good accordance with the recent rigorous homogenization results for convolution type operators  stating that, both in periodic and random stationary media, the effective
operator is a second order elliptic differential operator, see  \cite{BrPi21}, \cite{PiZh17}.

Since under the diffusive scaling convolution operators approximate  differential operators, it is natural
to compare the spectral properties of  non-local 
operators 
\eqref{L} and those of the classical    Schr\"odinger operators. It should be emphasized that, unlike the  Schr\"odinger operator,  the operator of multiplication by $V$ in \eqref{L} is not relatively compact with respect to the convolution operator. 
Therefore, the essential spectrum of $\cL$ need not coincide
with the essential spectrum of the convolution operator. And this is indeed the case: we show that the essential spectrum of $\cL$ coincides with the union of the essential spectrum of the multiplication by $V$ and that of the convolution operator, see Theorem~\ref{T1}.

It is well-known that localized perturbations of classical differential operators can create discrete eigenvalues below the essential spectrum. Since very first classical works  \cite{Simon2}, \cite{Klaus}, \cite{Klaus2}, \cite{Simon1}, such phenomenon was discovered and studied for plenty of models in hundreds of works and being not able to cite all of them, we just cite a recent book \cite{ExnKov}, in which a nice survey of the current state-of-art was provided. We also mention that a perturbation of the bottom of the essential spectrum for a differential operator can violate the preservation of total multiplicity and there are many mechanisms for such phenomenon, see \cite{GH1987}, \cite{BD1}, \cite{BD2}, 
\cite{BD4}, 
\cite{BD6}, \cite{BD7}. 
 It is then known that, in the case of perturbation of differential operators,  a weakly decaying potential can create infinitely many eigenvalues emerging from the bottom of the essential spectrum, see for instance \cite[Sect. X\!I\!I\!I.3]{RS4}. In dimension three an appropriate decay of the perturbing potential for such phenomenon is $|x|^{-2}$. In the non-local case that we study the potential and the convolution operator are both bounded and equipollent from this point of view. This is why in the considered case a much more wider class of localized perturbation can produce discrete eigenvalues bifurcating from the essential spectrum including the case of infinitely many such eigenvalues.
In the present work we focus on  sufficient conditions ensuring the existence of infinitely many eigenvalues that are formulated in terms of simple lower bounds
for  the potential $V$ at infinity and for the function $\hat a$ in the vicinity of its maximum point, see
Theorems \ref{T2}-\ref{thm_heavy_tail} in Section \ref{ss_disc_spe}.
It is worth noting that in the case of a non-negative  $a(\cdot)$  being a probability density the conditions of existence of an infinite discrete spectrum of  $\cL$ can be formulated in terms of the asymptotic behaviour of  $a(\cdot)$ and $V(\cdot)$ at infinity.
Then, in Section \ref{ss_gap},
we provide simple sufficient conditions for the presence of points of the discrete spectrum in spectral gaps.

\section{Main assumptions} 

Now we formulate general conditions on functions $V=V(x)$ and $a=a(x)$. We assume that
\begin{equation}\label{prop_aa}
a \in L_1(\mathds{R}^d), \qquad  a(-x)=\overline{a(x)}.
\end{equation}
Concerning $V$ we suppose that
\begin{itemize}
  \item $V$ is a real-valued function,
  \item  $V$ is an element of  $L_\infty(\mathbb R^d), $
\begin{equation}\label{ess_ogr}
-\infty<V_{\rm min}:=\mathop{\mathrm{essinf}}\limits_{x\in\mathbb R^d} V(x), \qquad
\mathop{\mathrm{esssup}}\limits_{x\in\mathbb R^d} V(x) =: V_{\max}<+\infty,
\end{equation}
  \item $V(x)$ tends to zero as $|x|\to 0$: for any $\delta>0$ there exists $N_\delta$ such that
\begin{equation}\label{vani_infi}
\mathop{\mathrm{esssup}}\limits_{|x|\geqslant N_\delta} |V(x)|\leqslant \delta.
\end{equation}
\end{itemize}
Due to \eqref{prop_aa}  the Fourier image $\hat{a}$ of the function $a$ is a continuous real-valued function with $\hat{a}(\lambda) \to 0$, as $|\lambda| \to \infty$, i.e. $\hat{ a} \in C_0(\mathds{R}^d)$.


Denote
$$a_{\rm min}:=\inf\limits_{\mathds{R}^d} \hat{a}, \quad a_{\max}:=\sup\limits_{\mathds{R}^d} \hat{a}.
$$
Then
\begin{equation}\label{aV0}
a_{\rm min}\leqslant 0\leqslant a_{\max}, \quad V_{\rm min}\leqslant 0\leqslant V_{\max}.
\end{equation}
In what follows  we also use the notation $\mu_0=\min\{a_{\rm min},V_{\rm min}\}$ and $\mu_1=\max\{a_{\max},V_{\max}\}$.

\section{Essential spectrum}
The spectrum of the operator of multiplication by $V$ in $L_2(\mathbb R^d)$ coincides with its essential spectrum
and is equal to the essential range of $V$. We denote the essential range of $V$ by $\mathcal{S}_V$.
The following theorem describe the essential spectrum of the operator $\cL$.

\begin{theorem}\label{T1}
Under the above formulated conditions on the functions $a$ and $V$ the essential spectrum of the operator
$\cL$ is the union
$$
\sigma_{\rm ess}(\cL)=[a_{\rm min},a_{\max}]\cup \mathcal{S}_V.
$$
The discrete spectrum of $\cL$  can be located only in the set
$$[a_{\rm min}+V_{\rm min}, a_{\max}+V_{\max}]\setminus \big([a_{\rm min},a_{\max}]\cup \mathcal{S}_V\big).$$
It can accumulate only to the boundary of $\big([a_{\rm min},a_{\max}]\cup \mathcal{S}_V\big)$.
\end{theorem}

\begin{proof}
First we show that $[a_{\rm min}, a_{\max}]\subset\sigma_{\rm ess}(\cL)$,
Assume that $\lambda\in (a_{\rm min}, a_{\max})$. Since $\hat a$ is continuous, there exists $\xi_0$ such that
  $\hat a(\xi_0)=\lambda$.  Consider  the sequence of functions
  $\hat\varphi_n(\xi)=\big(\frac n2\big)^{\frac d2}\mathbf{1}_{\xi_0+[-\frac1n,\frac1n]^d}$, $n=1,2,\ldots$ Observe that
  $\|\varphi_n\|_{L_2(\mathbb R^d)}=1$.

  Denote by $\mathcal{F}$ the Fourier transform in $\mathbb R^d$, and let $\varphi_n=\mathcal{F}^{-1}\hat \varphi_n$.
  Then, after direct computation, we have $\varphi_n=e^{-i\xi_0\cdot x}\big(\frac n2\big)^{\frac d2}\prod_{j=1}^{d}\frac 2 {x_j}\sin(\frac {x_j}n)$. Observe that $|\varphi_n|\leqslant \big(\frac2n\big)^{\frac d2}$. Due to \eqref{ess_ogr} and \eqref{vani_infi},
  \begin{align*}
   \|V\varphi_n\|^2_{L_2(\mathbb R^d)}= & \|V\varphi_n\|^2_{L_2([-n^\frac12,n^\frac12]^d)}+
   \|V\varphi_n\|^2_{L_2(\mathbb R^d\setminus[-n^\frac12,n^\frac12]^d)}
   \\
\leqslant & 4^d \|V\|^2_{L_\infty(\mathbb R^d)}n^{-\frac d2}+
   \mathop{\mathrm{esssup}}\limits_{x\in\mathbb R^d\setminus[-n^\frac12,n^\frac12]^d} |V(x)|^2\,\|\varphi_n\|_{L_2(\mathbb R^d)}
  \to\,0\quad\text{as}\quad n\to\infty.
  \end{align*}
We also have
$$
\|a\ast\varphi_n-\lambda\varphi_n\|^2_{L_2(\mathbb R^d)}=\|(\hat a-\lambda)\hat\varphi_n\|^2_{L_2(\mathbb R^d)} \to 0,\quad
\hbox{as }n\to\infty.
$$
Combining the last two limit relations yields
$$
\|(\cL - \lambda) \varphi_n\|_{L_2(\mathbb R^d)}\,\to\,0\quad\hbox{as }n\to\infty.
$$
Since the family $\{\varphi_n\}_{n=1}^\infty$ is not compact, this relation implies that $\lambda\in\sigma_{\rm ess}(\cL)$.

Assume now that $\lambda\in \mathcal{S}_V$. Then $|V^{-1}(\lambda-\frac1n,\lambda+\frac1n)|>0$ for any $n>0$, here and later on for a set $S\subset\mathbb R^d$ we denote by $|S|$ its Lebesgue measure.  Consider a Lebesgue
point $x_0$ of $V$ such that $x_0\in V^{-1}(\lambda-\frac1n,\lambda+\frac1n)$. By the definition of a Lebesgue point
there exists $\delta_n>0$ such that
$$
\int\limits_{ K_{\delta_n}(x_0)}|V(x)-V(x_0)|\,dx< \frac1n|K_{\delta_n}(x_0)|,
$$
here the symbol $K_\delta(x)$ stands for the cube $x+[-\delta,\delta]^d$.
We choose $\delta_n\leqslant \frac1n$ and let
$$
\phi_n=(|K_{\delta_n}(x_0)|^{-\frac12})\mathbf{1}_{K_{\delta_n}(x_0)}, \quad n=2,3,\ldots
$$
Then $\|\phi_n\|_{L_2(\mathbb R^d)}=1$ and
\begin{equation}\label{estimvcontr}
\begin{aligned}
\|(V-\lambda)\phi_n\|^2_{L_2(\mathbb R^d)}\leqslant & \frac2{|K_{\delta_n}\!(x_0)|}\int\limits_{ K_{\delta_n}\!(x_0)}|V(x)-V(x_0)|^2\,dx
\\
&+
 \frac2{|K_{\delta_n}\!(x_0)|}\int\limits_{ K_{\delta_n}\!(x_0)}|V(x_0)-\lambda|^2\,dx
\\
\leqslant& \frac{4\|V\|_{L_\infty(\mathbb R^d)}}{|K_{\delta_n}\!(x_0)|}\int\limits_{ K_{\delta_n}\!(x_0)}|V(x)-V(x_0)|\,dx+
\frac2{n^2}
\\
\leqslant & \frac1n\big(4\|V\|_{L_\infty(\mathbb R^d)}+2\big).
\end{aligned}
\end{equation}
One can easily calculate the Fourier transform of $\phi_n$:
$$
\hat\phi_n(\xi)=e^{i\xi\cdot x_0}(2\delta_n)^{-\frac d2}\prod_{j=1}^{d}\frac2{\xi_j}\sin(\delta_n\xi_j).
$$
From this formula we deduce that  $|\hat\phi_n(\xi)|\leqslant (2\delta_n)^{\frac d2}$. Since $\|\hat\phi_n\|_{L_2(\mathbb R^d)}=1$,
the $L_2$ norm of the convolution $a\ast\phi_n$ can be estimates as follows:
$$
\|a\ast\phi_n\|^2_{L_2(\mathbb R^d)}=\|\hat a\hat\phi_n\|^2_{L_2(\mathbb R^d)}=\int\limits_{K_{\delta_n^{-\frac12}}(0)}|\hat a|^2(\xi)|\hat \phi_n|^2(\xi)d\xi+\int\limits_{\mathbb R^d\setminus K_{\delta_n^{-\frac12}}(0)}|\hat a|^2(\xi)|\hat\phi_n|^2(\xi)d\xi
$$
$$
\leqslant \|\hat a\|^2_{L_\infty(\mathbb R^d)}4^d \delta_n^{\frac d2}+\max\limits_{\xi\in \mathbb R^d\setminus K_{\delta_n^{-\frac12}}(0)} |\hat a(\xi)| \longrightarrow 0 \quad\hbox{as }n\to\infty.
$$
Combining this inequality with \eqref{estimvcontr} we conclude that $\|(\cL-\lambda)\phi_n\|_{L_2(\mathbb R^d)}\to 0$.
Since by construction the family $\{\phi_n\}_{n=1}^\infty$ is not compact, this implies that $\lambda\in \sigma_{\rm ess}(\cL)$.
Therefore, $[a_{\rm min},a_{\max}]\cup \mathcal{S}_V\subset \sigma_{\rm ess}(\cL)$

The opposite inclusion $\sigma_{\rm ess}(\cL)\subset[a_{\rm min},a_{\max}]\cup \mathcal{S}_V$ can be justified in exactly the same way as in the proof of Theorem 2.1 in \cite{BPZ-JMAA}.  This completes the proof of the first statement of the theorem.

Since the quadratic form $(\cL u,u)$ satisfies an evident estimate
$$
\big(a_{\rm min}+V_{\rm min}\big) \|u\|^2_{L_2(\mathbb R^d)}\leqslant
(\cL u,u)\leqslant \big(a_{\max}+V_{\max}\big) \|u\|^2_{L_2(\mathbb R^d)},
$$
the spectrum of $\cL$ is situated in the interval $[a_{\rm min}+V_{\rm min},a_{\max}+V_{\max}]$, and, due to the first
statement of the theorem,  the discrete spectrum, if exists, occupies the segments $[a_{\rm min}+V_{\rm min},\mu_0)$ and
  $(\mu_1,a_{\max}+V_{\max}]$.   An accumulation point of the discrete spectrum of $\cL$ is an element of the essential
  spectrum of $\cL$. Therefore, this point must coincide either with $\mu_0$ or with $\mu_1$.
\end{proof}

\section{Discrete spectrum}\label{ss_disc_spe}
We turn now to the discrete spectrum of $\cL$ and consider its behaviour  in the segment $(\mu_1, a_{\max}+V_{\max}]$. In order to study the discrete spectrum in the segment $[a_{\rm min}+V_{\rm min},\mu_0)$ it suffices the exchange $\cL$ to $-\cL$.
It is clear that the necessary condition for the existence of a discrete spectrum in $(\mu_1, a_{\max}+V_{\max}]$ is the validity of the following inequality:
$$a_{\max}+V_{\max} > \max\{a_{\max}, V_{\max}\}.
$$
This inequality together with \eqref{aV0} imply that $a_{\max}>0$ and $V_{\max} > 0$.

We consider further the case $\mu_1=a_{\max}>0$ and $0< V_{\max}\leqslant a_{\max}$, and assume without loss of the generality  that  $a_{\max} = \hat a(0)$.

\begin{theorem}\label{T2}
Let $\mu_1=a_{\max}>0$, $0< V_{\max}\leqslant a_{\max}$, and assume that $a_{\max} = \hat a(0)$ and the following two conditions hold:

\noindent
1) there exist   constants $\alpha>0$, $\vartheta>0$ and $c>0$ such that
\begin{equation}\label{T-1-1}
\hat a(\xi)   \ge  a_{\max} - c | \xi|^\alpha 
\quad \text{for all}\quad  |\xi| \leqslant \vartheta;
\end{equation}

\noindent 2) there exist  constants $\gamma>0$,  $q>0$ and $C>0$ such that
\begin{equation}\label{T-1-2}
V(x)\ \ge \ C |x|^{-\gamma} \qquad \mbox{ for all } \; |x| \ge q.
\end{equation}

\noindent
If $ \alpha > \gamma $, then the operator $\cL$ has infinitely many eigenvalues in $(\mu_1, a_{\max}+V_{\max}]$  with the accumulation point $\mu_1$.
\end{theorem}
The conditions in \eqref{T-1-1}--\eqref{T-1-2} can be essentially relaxed, instead of point-wise estimates it is sufficient to assume
that weaker estimates in integral form are valid.   Denote
\begin{align*}
&
\langle V\rangle(R):=\frac1{\mathrm{meas}(G_R)}\int\limits_{G_R}V(x)\,dx,\qquad  G_R:=\{x\in\mathbb R^d\,:\,R\leqslant |x|
\leqslant 2R\},
\\
&
\langle \hat a\rangle(r):=\frac1{\mathrm{meas}(G_r)} \int\limits_{G_r} \hat a(\xi)d\xi,
\end{align*}

The following statement holds.
\begin{theorem}\label{th_enforces}
   Assume that $\mu_1=a_{\max}>0$, $0< V_{\max}\leqslant a_{\max}$, and  $a_{\max} = \hat a(0)$. Assume,
   moreover, that the following two conditions hold:

\noindent
1) there exist   constants $\alpha>0$,  $\vartheta>0$ and $c>0$ such that
\begin{equation}\label{T-1-1bis}
\langle\hat a\rangle(r) \ \ge \ a_{\max} - c  r^\alpha 
\quad \mbox{ for all  } \;  r \leqslant \vartheta;
\end{equation}

\noindent 2) there exist  constants $\gamma>0$, $q>0$ and $C>0$ such that
\begin{equation}\label{T-1-2bis}
\langle V\rangle(R)\ \ge \ C R^{-\gamma} \qquad \mbox{ for all } \; R \ge q.
\end{equation}

\noindent
If $ \alpha > \gamma $, then the operator $\cL$ has infinitely many eigenvalues in $(\mu_1, a_{\max}+V_{\max}]$, with the accumulation point $\mu_1$.
\end{theorem}


Clearly, the statement of Theorem \ref{T2} follows from that of Theorem \ref{th_enforces}. Therefore, it suffices to prove
Theorem \ref{th_enforces}. 

\begin{proof}[Proof of Theorem \ref{th_enforces}]
Our goal is to
construct a countable family of functions such that the  quadratic form of the operator $\cL - \mu_1$ is positive definite on the linear span of these functions.


Let $\psi \in C_0^\infty(\mathds{R}^d)$ be an infinitely differentiable radially symmetric real positive function such that
\begin{itemize}
\item $\supp \psi \subset \{ \frac12<|x|<\frac52 \}$,
\item $\psi(x)\equiv h = const$ for $x\in \{1<|x|<2 \}$,
\item $0\leqslant \psi(x)\leqslant h$ for all $x\in\mathbb R^d$,
\item $\|\psi\|_{L_2(\mathds{R}^d)}=1$.
\end{itemize}
Observe that in this case the Fourier transform of $\psi$ is also radially symmetric and real.
For an arbitrary $R>0$ denote  $\psi_R(x) = R^{-d/2}\psi\big( \frac{x}{R} \big)$. Then
\begin{equation*}
\supp \psi_R \subset \left\{ \frac12 R <|x|<\frac52 R\right\}\quad \text{and}\quad  \|\psi_R\|_{L_2(\mathds{R}^d)}=1.
\end{equation*}
For the Fourier transform $\hat \psi(\xi)$ of the function $\psi(x)$ we have $\hat \psi_R(\xi) = R^{d/2} \hat \psi (R\xi)$.

The quadratic form of the operator  $\cL - \mu_1$ reads
\begin{equation}\label{QF}
  \big( (\cL - \mu_1)\psi_R,\psi_R \big) = \big( (\hat a - a_{\max}) \hat \psi_R, \hat \psi_R \big) + \big( V \psi_R, \psi_R \big).
\end{equation}
 Taking sufficiently large $R$ we estimate separately each term on the right-hand side of  \eqref{QF}.
It follows from  condition 2) of the theorem that
\begin{equation}\label{V}
\begin{aligned}
 \big( V \psi_R, \psi_R \big) \geqslant  &\int\limits_{R\leqslant |x|\leqslant 2R} R^{-d}\psi^2 \left( \frac{x}{R} \right)\, V(x)\,dx
 =R^{-d}h^2 \, \int\limits_{R\leqslant |x|\leqslant 2R}V(x)\,dx
 \\
 \geqslant &h^2 \mathrm{meas}(G_1) \langle V\rangle(R)
\geqslant  C h^2 \mathrm{meas}(G_1) R^{-\gamma}.
\end{aligned}
\end{equation}
Let us estimate the first term on the right-hand side of \eqref{QF}.
Since $\psi \in C_0^\infty(\mathds{R}^d)$, then $\hat \psi(\xi)$ is a function of the Schwarz class. Therefore,
$$
\hat \psi (\xi) \ |\xi|^k \ \to 0, \qquad |\xi| \to \infty \qquad \forall k=1,2, \ldots.
$$
Consequently,
$$
\max\limits_{r<|\xi|<2r} |\hat \psi (\xi)|^2 (1+r)^{d+\alpha+1}\leqslant C \quad\text{for all}\quad r>0.
$$
By this relation, Condition 1) of the theorem and the boundedness of $\hat a(\cdot)$, we obtain
\begin{equation}\label{A_bis0}
\begin{aligned}
R^d\int\limits_{|\xi|\leqslant \vartheta} |\hat \psi(R\xi)|^2 \, \big( \hat a (\xi) - a_{\max} \big)d \xi
=&R^d\sum\limits_{j=1}^\infty \int\limits_{G_{2^{-j}\vartheta}}
 |\hat \psi(R\xi)|^2 \, \big( \hat a (\xi) - a_{\max} \big)d \xi
\\
\geqslant & R^d\sum\limits_{j=1}^\infty\max\limits_{G_{(2^{-j}\vartheta R)}}|\hat \psi(\xi)|^2 \!\!\!
 \int\limits_{G_{2^{-j}\vartheta}} \!\!\! \big( \hat a (\xi) - a_{\max} \big)d \xi
\\
\geqslant & -R^d\sum\limits_{j=1}^\infty \frac {C_1 \mathrm{meas}(G_{(2^{-j}\vartheta)})\,(2^{-j}\vartheta)^\alpha}
 {(1+2^{-j}\vartheta R)^{d+\alpha+1}}
 \\
=&-R^{-\alpha}\sum\limits_{j=1}^\infty \frac {C_2 (2^{-j}\vartheta R)^{d+\alpha}}
 {(1+2^{-j}\vartheta R)^{d+\alpha+1}}
  \\
 =&-R^{-\alpha}\Big(\sum\limits_{j=1}^{j_0} \frac {C_2 (2^{-j}\vartheta R)^{d+\alpha}}
 {(1+2^{-j}\vartheta R)^{d+\alpha+1}}
 \\
 &+\sum\limits_{j=j_0+1}^{\infty} \frac {C_2 (2^{-j}\vartheta R)^{d+\alpha}}
 {(1+2^{-j}\vartheta R)^{d+\alpha+1}} \Big)
  \\
 \geqslant & -R^{-\alpha}\Big(\sum\limits_{j=1}^{j_0} \frac {C_2}
 {2^{-j}\vartheta R} +\sum\limits_{j=j_0+1}^{\infty} C_2 (2^{-j}\vartheta R)^{d+\alpha} \Big)
 \\
 \geqslant & -C_3 R^{-\alpha},
\end{aligned}
\end{equation}
where $j_0\in\mathbb N$ is such that $1\leqslant 2^{-j_0}\vartheta R<2$.
Since $\hat\psi(\cdot)$  is a Schwarz class function, for $R>1$ the integral over the set $\{|\xi|>\vartheta\}$ can be estimated
as follows:
\begin{equation}\label{A_bis00}
\begin{aligned}
R^d\int\limits_{|\xi|> \vartheta} |\hat \psi(R\xi)|^2 \, \big( \hat a (\xi) - a_{\max} \big)d \xi\geqslant &
 -2\|\hat a\|_{L_\infty(\mathbb R^d)}
  \int\limits_{|\xi| > \vartheta R} |\hat\psi(\xi)|^2  d\xi
  \\
  \geqslant & - c_4 (\vartheta R)^{-2\alpha}.
\end{aligned}
\end{equation}
Combining \eqref{A_bis0} and  \eqref{A_bis00} yields
\begin{equation}\label{A}
\begin{aligned}
 \big( (\hat a - a_{\max}) \hat \psi_R, \hat \psi_R \big) = &R^d\int\limits_{\mathds{R}^d} |\hat \psi(R\xi)|^2 \, \big( \hat a (\xi) - a_{\max} \big)d \xi
  \\
=& R^d\int\limits_{|\xi|\leqslant \vartheta} |\hat \psi(R\xi)|^2 \, \big( \hat a (\xi) - a_{\max} \big)d \xi
\\
&+
R^d\int\limits_{|\xi|>\vartheta} |\hat \psi(R\xi)|^2 \, \big( \hat a (\xi) - a_{\max} \big)d \xi
  \\
\geqslant &-C_3  R^{-\alpha}-c_4 (\vartheta R)^{-2\alpha}> -C_5 R^{-\alpha}
\end{aligned}
 \end{equation}
 for $R>1$.
Thus, for $\gamma < \alpha$ it follows from \eqref{QF} - \eqref{A} that there exist $R_0>0$ and $c_2>0$ such that
$$
\big( (\cL - \mu_1)\psi_R,\psi_R \big) \geqslant  \frac12C h^2 \mathrm{meas}(G_1)  R^{-\gamma},
$$
if $R\geqslant R_0$.

Next we should prove that quadratic form \eqref{QF} is positive definite on the linear span of a countable set of the functions of the form $\psi_R$. Let us first assume that $a(z)$ has a finite support.
Taking $\varphi_n = \psi_{2^{2n-1}R}, \ n=1,2,\ldots,$ where $R>1$ is large enough, we conclude that
$$
\big( (\cL - \mu_1)\varphi_n, \varphi_n \big)> 0 \quad \mbox{ for all } \; n=1,2, \ldots.
$$
Moreover,
$$
\big( (\cL - \mu_1)\varphi_n, \varphi_m \big)= 0 \quad \mbox{ if } \;  n \neq m,
$$
since the supports of the functions $ \int\limits_{\mathds{R}^d} a(x-y)\varphi_n (y) dy$ and $\varphi_m$ do not intersect for large $R$. Consequently, the quadratic form of the operator $(\cL - \mu_1)$ is positive definite on the linear span of $\{\varphi_n\}$, and thus the operator $\cL $ has infinitely many eigenvalues to the right of the edge $\mu_1$.

If $\supp  a(\cdot)$ is not compact,
then we take
$$\varphi_n = \psi_{2^{Mn}R},
$$
the constant $M>2$ will be specified later on.
Since the supports of functions $\varphi_n$ and $\varphi_m$ do not intersect for $n \neq m$, the same arguments as those used in estimate \eqref{A} yield for $n<m$ the following bound
\begin{align*}
|\big( (\cL - \mu_1)\varphi_n, \varphi_{m} \big)| = & \left|\big( (\hat a - a_{{\max}}) \hat \varphi_n, \hat \varphi_m \big)\right|
\\
=& \big( 2^{Mn} \, R \big)^{d/2} \,  \big( 2^{M m} \, R \big)^{d/2}
  \left|\int\limits_{\mathds{R}^d} ( a_{{\max}} - \hat a(\xi)) \, \hat \psi (2^{Mn} \, R\xi)
   \hat \psi (2^{Mm} \, R\xi)\, d\xi\right|
\\
\leqslant &\bigg(\!\big( 2^{M n}  R \big)^{d}\! \int\limits_{\mathbb R^d}( a_{{\max}} - \hat a(\xi)) \, |\hat \psi (2^{Mn} \, R\xi)|^2d\xi\bigg)^{\!\frac12}
  \\
  &\cdot\bigg(\!\big( 2^{M m}  R \big)^{d}\! \int\limits_{\mathbb R^d}( a_{{\max}} - \hat a(\xi)) \, |\hat \psi (2^{Mm} \, R\xi)|^2d\xi\bigg)^{\frac12}
\\
\leqslant & C_5 \big( 2^{M n}  R \big)^{-\frac\alpha2}\big( 2^{M m}  R \big)^{-\frac\alpha2}
\\
=&C_5
 \big( 2^{M n}  R \big)^{-\frac\gamma2}\big( 2^{M m}  R \big)^{-\frac\gamma2}
 \big [\big( 2^{M n}  R \big)^{\frac{\gamma-\alpha}2}\big( 2^{M m}  R \big)^{\frac{\gamma-\alpha}2}\big]
\end{align*}
 It remains to choose sufficiently large $M$
so that for all $R\geqslant R_0$ it holds
$$
C_5 \, \big( 2^{M n}  R \big)^{\frac{\gamma-\alpha}2}\big( 2^{M m}  R \big)^{\frac{\gamma-\alpha}2}
< 4^{-(n+m)}\frac12C h^2 \mathrm{meas}(G_1) ,
\quad m,\,n=1,2,\ldots
$$
Then $$|\big( (\cL - \mu_1)\varphi_n, \varphi_{m} \big)|< 4^{-(m+n)}\big( (\cL - \mu_1)\varphi_m, \varphi_{m} \big)^\frac12
\big( (\cL - \mu_1)\varphi_n, \varphi_{n} \big)^\frac12,$$ and the desired positive definiteness
of the quadratic form $\big( (\cL - \mu_1)\psi, \psi \big)$  on the linear span of functions $\{\varphi_n\}_{n=1}^\infty$ follows.
\end{proof}

In the models of population dynamics the function $a$ is a probability density while the potential $V$
satisfies the inequalities $0\leqslant V\leqslant 1$, see \cite{KPMZh}.
If  $a(z)=a(|z|)\ge 0$ is a probability density, then $a_{\max}=\hat a(0) =1$. Assume that either
\begin{equation}\label{R1}
  m = \int\limits_{\mathbb{R}^d} |z|^2 \, a(z) \, dz < \infty,
\end{equation}
or
\begin{equation}\label{R2}
  a(z) \sim \frac{c_1}{|z|^{\alpha +d}} \quad \mbox{ as  } \; |z| \to \infty, \quad \mbox{with } \; 0<\alpha < 2.
\end{equation}
Then \eqref{R1} together with estimate $\sin^2 x \le x^2$ yield that
$$
1- \hat a(\xi) = \int\limits_{\mathbb{R}^d} a(z) (1 - \cos {z \xi}) \, dz =  2 \int\limits_{\mathbb{R}^d} a(z) \sin^2 \frac{z \xi}{2} \, dz \le \frac{m}{2} \xi^2,
$$
and hence
$$
\hat a(\xi) \ge 1 -  \frac{m}{2} \xi^2.
$$
The relation \eqref{R2} implies that
$$
\hat a (\xi) = 1 - c |\xi|^\alpha + o(|\xi|^\alpha), \quad \mbox{ as  } \; |\xi| \to 0,
$$
see e.g. \cite{SK}. 

\begin{corollary}
Suppose that $0< V_{\max}\le 1$ and probability density $a(z)$ satisfies one of  conditions \eqref{R1}, \eqref{R2}. Then by Theorem~\ref{T2}, the operator  $\cL$ has infinitely many eigenvalues in $(1, 1+V_{\max}]$, if
$$\liminf\limits_{|x|\to\infty}\frac{V(x)}{|x|^{-\gamma}}>0
$$  for some $0<\gamma<2$ in the case \eqref{R1} and for $0<\gamma<\alpha$ in the case \eqref{R2}.
\end{corollary}

Thus, for the probability distribution with the density  $a(z)$  the existence of an infinite discrete spectrum of the operator $\cL$ is determined by the asymptotic behaviour of the tails of the functions $a(\cdot)$ and $V(\cdot)$ at infinity. In particular, this statement complements the results on the existence of a positive discrete spectrum of the so-called non-local Schr{\"o}dinger operators presented in \cite{KPMZh, KMV}.

\bigskip

The case $a_{\max}=\hat a(\xi_0)$ with $\xi_0 \neq 0$ can be treated in a similar way.
We introduce a test function $\psi$ as in the proof of Theorem \ref{T2} and define $\psi_R(\cdot)$ by
$$
\psi_R(x) = R^{- d/2}\, e^{i x \xi_0} \, \psi\left(\frac{x}{R}\right).
$$
Then $\hat \psi_R(\xi) = R^{d/2} \, \hat \psi( R(\xi - \xi_0))$.
\begin{remark}
With evident modifications in the proof, condition \eqref{T-1-2bis} in the formulation of Theorem \ref{th_enforces} can be replaced with a weaker one that reads
$$
\limsup\limits_{R\to\infty} \frac{\langle V\rangle(R)}{R^{-\gamma}}\,>0.
$$
Indeed, in this case there exists $\theta>0$ and a sequence $R_j\to\infty$ as $j\to\infty$ such that $\langle V\rangle(R)\geqslant \theta R^{-\gamma}_j$. Then  one can take $\psi_n(x)=\psi(MRz_n\xi)$ with $z_n$
that satisfy the following conditions:
\begin{equation*}
 z_{n+1}> z_n+1\quad \text{and}\quad \langle V\rangle(MRz_n)\geqslant \theta(MRz_n)^{-\gamma}.
\end{equation*}
\end{remark}

\medskip

In the case of potentials having heavy tails we introduce an additional characteristic of the kernel $a$:
\begin{equation}\label{char_a_ht}
  \ell_{\hat a}(r)=\int\limits_{\mathbb R^d}(a_{\max}-\hat a(r\xi))e^{-\xi^2}d\xi.
\end{equation}
\begin{theorem}\label{thm_heavy_tail}
  Let  $\langle V\rangle (R)$ admit the lower bound $\langle V\rangle (R)\geqslant R^{-\frac d4}$ for all sufficiently large
$R$, and assume that  $\mu_1=a_{\max}>0$,   $a_{\max} = \hat a(0)$ and
\begin{equation}\label{a_V_ineq}
\frac{ \ell_{\hat a}(R^{-1})}{\langle V\rangle (R)} \to 0\quad \text{as}\quad R \to\infty.
\end{equation}
Then the operator $\cL$ has infinitely many eigenvalues in $(\mu_1, a_{\max}+V_{\max}]$, with the accumulation point $\mu_1$.
\end{theorem}
\begin{proof}
  We consider a sequence of Gaussian test functions $\varphi_{R_j}(x)=R_j^{-\frac d2}e^{-\frac{x^2}{R^2_j}}$,
  where the scaling factors $R_j>0$, $j=1,2,\ldots$ will be chosen later on.
  Observe that
\begin{equation}\label{gau_bound1}
  \int\limits_{\mathbb R^d}  V(x)(\varphi_R(x))^2\,dx=\int\limits_{\mathbb R^d}V(x)R^{-d}e^{-\frac{2x^2}{R^2}}\,dx\geqslant
  C_1\, \langle V\rangle\!(R),
\end{equation}
where the constant $C_->0$ does not depend on $R$. Since $V(\cdot)$ is bounded, for any integer $j>1$ we also have
\begin{equation}\label{gau_bound2}
  \int\limits_{\mathbb R^d}  V(x)(\varphi_R(x))\varphi_{R^j}(x)\,dx\leqslant C \int\limits_{\mathbb R^d}\varphi_R(x)\varphi_{R^j}(x)\,dx \leqslant C R^{-\frac{d(j-1)}2};
\end{equation}
hereinafter $C$  is a positive constant independent of $R$, but it can change its value from one formula to another.
Denote
  $$
  \theta_{kj}:=\int\limits_{\mathbb R^d}V(x)\varphi_{R_k}(x) \varphi_{R_j}(x)\,dx.
  $$
 Letting $R_{j+1}=R^4_j$, j=1,2,\ldots, by \eqref{gau_bound1}-\eqref{gau_bound2} and taking into account the lower bound $\langle V\rangle (R)\geqslant R^{-\frac d4}$, for $k>j$ we obtain
$$
(\theta_{kj})^2\leqslant CR_j^{- (2^{2(k-j)}-1)d}\leqslant R_j^{-\frac d4}R_k^{-\frac d4} R_j^{-\frac d42^{2(k-j)}}
\leqslant C R_j^{-\frac d42^{2(k-j)}}\theta_{jj}\theta_{kk}.
$$
 Choosing sufficiently large $R_1$
 we conclude that for any function $\varphi\in L^2(\mathbb R^d)$ such that
 $\varphi=\sum\limits_{j=1}^N \kappa_j\varphi_{R_j}$ the following inequality holds:
 \begin{equation}\label{est_v_solo}
   \int\limits_{\mathbb R^d} V(x)\varphi^2(x)\,dx\geqslant \frac12\sum\limits_{j=1}^N\kappa_j^2\int\limits_{\mathbb R^d} V(x)
   (\varphi_{R_j}(x))^2\,dx\geqslant \frac{C_-}2\sum\limits_{j=1}^N\kappa_j^2 \langle V\rangle(R_j).
 \end{equation}
 Since the Fourier transform of $\varphi_{R_j}$ is $\hat\varphi_{R_j}(\xi)=c_d R_j^{\frac d2}
 e^{-R_j^2\xi^2}$ with $c_d=(2\pi)^{-\frac d2}$, the quantities   $\big((a_{\max}-\hat a)\hat\varphi_{R_j},
 \hat\varphi_{R_j}\big)_{L^2(\mathbb R^d)}$ and $\big((a_{\max}-\hat a)\hat\varphi_{R_j},
 \hat\varphi_{R_k}\big)_{L^2(\mathbb R^d)}$ with $k>j$ can be estimated as follows:
 \begin{align*}
&
c_d^2 \int\limits_{\mathbb R^d}(a_{\max}-\hat a(\xi))R^d_j e^{-2R_j^2\xi^2}d\xi=
c_d^2 \int\limits_{\mathbb R^d}\left(a_{\max}-\hat a\left(\frac\xi{R_j}\right)\right) e^{-2\xi^2}d\xi\leqslant c_d^2\ell_a({R_j^{-1}}),
\\
&c_d^2 \int\limits_{\mathbb R^d}(a_{\max}-\hat a(\xi))R^{\frac d2}_j R^{\frac d2}_k
e^{-R_j^2\xi^2}e^{-R_k^2\xi^2}d\xi=
c_d^2 \Big(\frac{R_j}{R_k}\Big)^{\frac d2}\int\limits_{\mathbb R^d}
\Big(a_{\max}-\hat a\left(\frac\xi{R_k}\right)\Big) e^{-\xi^2}
e^{-\frac{R_j^2\xi^2}{R_k^2}}d\xi
\\
&\hphantom{c_d^2 \int\limits_{\mathbb R^d}(a_{\max}-\hat a(\xi))R^{\frac d2}_j R^{\frac d2}_k
e^{-R_j^2\xi^2}e^{-R_k^2\xi^2}d\xi}
\leqslant c_d^2R_j^{-(2^{2(k-j)}-1)\frac d2}\ell_a({R_k^{-1}}).
 \end{align*}
 In view of \eqref{a_V_ineq} for sufficiently large $R_ 1$ this implies the estimate
 \begin{align*}
 \big((a_{\max}-\hat a)\hat\varphi, \hat\varphi\big)_{L^2(\mathbb R^d)}\leqslant &
 \frac{C_-}{20}\sum\limits_{j=1}^N\kappa_j^2 \langle V\rangle(R_j)+ \frac{C_-}{20}\sum\limits_{j=1}^N
 \sum\limits_{k=j+1}^N\kappa_j\kappa_k R_j^{(2^{2(k-j)}-1)\frac d2} \langle V\rangle(R_k)
 \\
 \leqslant & \frac{C_-}{20}\sum\limits_{j=1}^N\kappa_j^2 \langle V\rangle(R_j)
 \\
 &+ \frac{C_-}{20}\sum\limits_{j=1}^N
 \sum\limits_{k=j+1}^N\kappa_j\kappa_k R_j^{(2^{2(k-j)}-\frac32)\frac d2} \big(\langle V\rangle(R_k)\big)^{\frac12}
 \big(\langle V\rangle(R_j)\big)^{\frac12}
 \\
\leqslant &  \frac{C_-}{10}\sum\limits_{j=1}^N\kappa_j^2 \langle V\rangle(R_j);
 \end{align*}
 here we have also used the inequality $\langle V\rangle (R)\geqslant R^{-\frac d4}$.
 Combining this estimate with \eqref{est_v_solo} we obtain
 $$
 \big((\mathcal{L}-\mu_1)\varphi,\varphi\big)\geqslant  \frac{C_-}{4}\sum\limits_{j=1}^N\kappa_j^2 \langle V\rangle(R_j)\geqslant c_N\|\varphi\|^2_{L^2(\mathbb R^d)}, \quad c_N>0.
 $$
 This yields the desired statement.
\end{proof}

We turn to the case  $\mu_1=V_{\max}>0$. In this case the statements similar to those of Theorems
\ref{T2}--\ref{thm_heavy_tail} remain valid, if we exchange the roles of $V$ and $\hat a$.
For the reader convenience we formulate here the counterpart of Theorem \ref{T2}.

\begin{theorem}\label{T3}
Let $\mu_1=V_{\max} = V(0) >0$, $\ 0< a_{\max}\leqslant V_{\max}$, and assume that

\noindent 1) there exists a small enough $\vartheta>0$ such that
$$
V(x) \ \ge \ V_{\max} - c | x |^{\gamma}  \quad \mbox{for all} \quad |x| < \vartheta,
$$

\noindent 2) there exists a large enough $q>0$ such that
$$
\hat a(\xi)\ \ge \ C |\xi|^{-\alpha}  \quad \mbox{for all} \quad |\xi| \ge q,
$$ where  $c,\, C$ are constants.

\noindent If $\gamma > \alpha$, then the operator $\cL$ has infinitely many eigenvalues in $(\mu_1, \, \mu_1+ a_{\max}]$, with the accumulation point $\mu_1$.
\end{theorem}

\begin{proof}
The proof of of Theorem \ref{T3} is analogous to that of Theorem \ref{T2}. We consider a family of test functions
$$
\hat \psi_R (\xi) = R^{-d/2} \hat \psi \left(\frac{\xi}{R}\right) \qquad \mbox{with} \qquad  \supp \hat \psi \subset \{ q<|\xi|< 2q\}.
$$
In the same way as  inequalities  \eqref{V}  and \eqref{A} were derived we obtain the estimates
$$
 \big( \hat a \hat \psi_R, \hat \psi_R \big) \ge C\int\limits_{\mathds{R}^d} |\xi|^{-\alpha} \hat \psi^2 \left(\frac{\xi}{R}\right)\,  d \frac{\xi}{R} \ge  R^{- \alpha} \, \big( C |\xi|^{-\alpha} \hat \psi(\xi), \hat \psi(\xi) \big),
$$
and 
$$
\big( (V- V_{\max})\psi_R, \psi_R \big) \ge - R^{-\gamma} c \int\limits_{\mathds{R}^d}  |\psi(y)|^2 |y|^\gamma d y.
$$
The rest of the proof follows the line of the proof of Theorem \ref{T2}.
\end{proof}

\section{Discrete spectrum in spectral gaps}\label{ss_gap}

In this section we construct convolution operators with a potential that have a non-empty discrete spectrum in spectral gaps.
Consider an operator $\cL$ defined in \eqref{L} with  $V=V^0+V^1$,  and assume that the functions  $a(\cdot)$
and $V^0(\cdot)$ satisfy all the conditions of Theorem \ref{th_enforces} or \ref{thm_heavy_tail}.  Assume moreover that $V^1$ is a bounded function
with a compact support. Letting $\tilde{\mathcal{S}}^1$ be the essential range of $V^1+V^0 \mathbf{1}_{\supp V^1}$ we denote
$\mathcal{S}^1=\tilde{\mathcal{S}}^1\setminus\{0\}$, \  $\theta_-=\inf ({\mathcal{S}}^1)$ and
$V^1_{\max}=\mathrm{esssup} (V^1)$. In the sequel we suppose $\theta_->a_{\max}$.
Observe that in this case $a_{\max}$ and $\theta_-$ belong to the essential spectrum of $\cL$,
and $(a_{\max},\theta_-)$ is a gap in the essential spectrum.

Under the above formulated assumptions,  
in the set $\{\lambda\in\mathbb R\,:\, \lambda>a_{\max}\}$ there is a countable sequence of eigenvalues of the operator $\cL^0 u:=a\ast u+V^0u$ that converges to $a_{\max}$. We denote by $\lambda_0$ the largest
of them and by $u_0$ the corresponding eigenfunction.

\begin{theorem}\label{T4}
  Let the functions $a(\cdot)$ and $V^0(\cdot)$ satisfy all the conditions of Theorem \ref{th_enforces} or \ref{thm_heavy_tail}, and assume that
  $\theta_->\lambda_0$. Then there exists $\varkappa_0>0$, $\varkappa_0=\varkappa_0(a_{\max}, V_{\max}^1, \theta_-, \lambda_0, u_0(\cdot))$,
  such that if $|\supp V^1|<\varkappa_0$, then the operator $\cL$ has an eigenvalue in
  the interval $(a_{\max},\theta_-)$.
\end{theorem}

\begin{proof}
  We first calculate the $L_2$ norm of the function $(\cL u_0-\lambda_0 u_0)$.
\begin{align*}
  \|\cL u_0-\lambda_0 u_0\|_{L_2(\mathbb R^d)}\leqslant & \|\cL^0 u_0-\lambda_0 u_0\|_{L_2(\mathbb R^d)}
  +\| V^1 u_0\|_{L_2(\mathbb R^d)}=\| V^1 u_0\|_{L_2(\mathbb R^d)}
  \\
  \leqslant & \|V^1\|_{L_\infty}\Bigg(\,\int\limits_{\supp V^1}u^2_0(x)\,dx\Bigg)^\frac12.
\end{align*}
Since $u_0^2(\cdot)$ is integrable, there exists $\varkappa_0>0$ such that
\begin{equation}\label{vish_integr}
\Bigg(\,\int\limits_{\supp V^1}u_0^2(x)\,dx\Bigg)^\frac12 < \big(\|V^1\|_{L_\infty}\big)^{-1}\min\{\lambda_0-a_{\max},\,\theta_--\lambda_0\},
\end{equation}
if $|\supp V^1|\leqslant \varkappa_0$. For such $V^1$ we have
$$
\delta:=\|\cL u_0-\lambda_0 u_0\|_{L_2(\mathbb R^d)}\leqslant \min\{\lambda_0-a_{\max},\,\theta_--\lambda_0).
$$
By \cite[Lemma 12]{ViLu} we conclude that  there exists a point of the discrete spectrum of $\cL$ in the interval
$(\lambda_0-\delta,\lambda_0+\delta)$. Since $(\lambda_0-\delta,\lambda_0+\delta)\subset (a_{\max},\theta_-)$,
this implies the desired statement.
\end{proof}

\begin{remark}
{\rm   In a similar way, taking sufficiently small $\varkappa_0$, one can show that for any finite collection of distinct eigenvalues of the operator $\cL^0$ there exists an eigenvalue of the operator $\cL$ in a small neighbourhood of each element of this collection. }
\end{remark}

\section*{Acknowledgments}

D.I.B. is partially supported by the Czech Science Foundation within the project 22-18739S.
The work of the second and third authors has been partially supported by the Troms{\o} Research Foundation, project "Pure
Mathematics in Norway" and UiT Aurora project MASCOT.


\begin{thebibliography}{12}

\bibitem{Rossi_book} F. Andreu-Vaillo, J.M. Mazon, J.D. Rossi, J.J. Toledo-Melero, {\sl Nonlocal Diffusion Problems},
AMS,  Providence, RI (2010).


\bibitem{BPZ-JMAA}
D. Borisov, A. Piatnitski, E. Zhizhina, On the spectrum of convolution operator with a potential // \textit{J. Math. Anal. Appl.}, {\bf 517}:1, 126568  (2023).

\bibitem{BrPi21}
A. Braides, A.;   A. Piatnitski,   {\it    Homogenization of random convolution energies} //  J. London Math. Soc.   {\bf 10}:2, 295--319 (2021).

\bibitem{FKK} D. Finkelshtein, Yu. Kondratiev, O. Kutoviy, {\it Individual Based Model with Competition in Spatial Ecology} // SIAM J. Math. Anal. 41:1, 297-317 (2009).

\bibitem{KPZ} Yu. Kondratiev, S. Pirogov, E. Zhizhina, {\it A Quasispecies Continuous Contact Model
in a Critical Regime} // J. Stat. Phys. {\bf 163}:2, 357--373 (2016).

\bibitem{KKP} Yu. Kondratiev, O. Kutoviy, S. Pirogov, {\it Correlation functions
and invariant measures in continuous contact model} // Infin. Dimens. Anal. Quantum Probab. Relat. Top. {\bf 11}:2, 231--258 (2008).

\bibitem{KPMZh} Yu. Kondratiev, S. Molchanov, S. Pirogov, E. Zhizhina, {\it On ground state of some non-local Schr\"{o}dinger operators} // Appl. Anal. {\bf 96}:8, 1390--1400  (2017).

\bibitem{KMV} Yu. Kondratiev, S. Molchanov, B. Vainberg, {\it Spectral analysis of non-local Schr\"{o}dinger operators} // J. Func. Anal.  {\bf 273}:3,  1020--1048 (2017).

\bibitem{SK} L.B. Koralov, Ya.G. Sinai, {\it Theory of Probability and Random Processes}, Springer, Berlin (2007).

\bibitem{PiZh17} A. Piatnitski, E. Zhizhina,  {\it  Periodic homogenization of nonlocal operators with a convolution-type kernel} // {\it   SIAM J. Math. Anal.}  {\bf 49}:1,  64--81 (2017).


\bibitem{PZ} S. Pirogov, E. Zhizhina, {\it General contact processes: inhomogeneous models, models on graphs and on
manifolds} // Elect. J. Prob. {\bf 27}, 41 (2022). 

\bibitem{RS4} M. Reed, B. Simon, {\it Methods of Modern Mathematical Physics, V. IV: Analysis of operators}, Academic Press, London (1978).

\bibitem{ViLu}  M.I. Vishik, L.A. Lusternik,  Regular degeneration and boundary layer for linear
differential equations with small parameter // in ``Six Papers on Partial Differential Equations'',  
Amer. Math. Soc. Transl. Ser. 2. {\bf 20}, 239--364 (1962).


\bibitem{Simon2} R. Blankenbecler,
M. L. Goldberger, and B. Simon, {\it The bound states of weakly coupled long-range one-dimensional quantum Hamiltonians} // {\it Ann. Phys.} {\bf 108}, 69--78 (1977).

\bibitem{Klaus} M. Klaus, {\it  On the bound state of Schr\"odinger operators in one dimension} // {\it Ann. Phys.} {\bf 108}, 288--300 (1977).

\bibitem{Klaus2} M. Klaus and B. Simon,  {\it Coupling constant thresholds in nonrelativistic quantum mechanics, I. Short range two-body case} // {\it Ann. Phys.}  {\bf 130},
251--281 (1980).

\bibitem{Simon1} B. Simon, {\it The bound state of weakly coupled Schr\"odinger operators in one and two dimensions} // {\it Ann. Phys.}  {\bf 97}, 279--288 (1976).


\bibitem{ExnKov} P. Exner, H. Kova\v{r}\'{\i}k.
{\it Quantum Waveguides}, Springer, Cham (2015)


\bibitem{GH1987} F. Gesztesy and H. Holden, {\it A unified approach to eigenvalues and resonances of Schr\"odinger
operators using Fredholm determinants} // J. Math. Anal. Appl. {\bf 123}:1, 181--198 (1987).

\bibitem{BD1}  D. Borisov, {\it Perturbation of threshold of essential spectrum for waveguide with window. I. Decaying resonance solutions} // J. Math. Sci. {\bf 205}:2, 19--54
    (2015).

\bibitem{BD2} D.I. Borisov, D.A. Zezyulin, {\it Sequences of closely spaced resonances and eigenvalues for bipartite complex potentials} // Appl. Math. Lett. {\bf 100}, 106049 (2020).




\bibitem{BD4} D.I. Borisov, D.A. Zezyulin, M. Znojil, {\it  Bifurcations of thresholds in essential spectra of elliptic operators under localized non-Hermitian perturbations} // Stud.  Appl. Math. {\bf 146}:4, 834--880 (2021).



\bibitem{BD6} D.I. Borisov and D.A. Zezyulin, {\it Bifurcations of essential spectra generated by a small non-Hermitian small hole. II. Eigenvalues and resonances} // Russ. J. Math. Phys. {\bf 29}:3, 321--341 (2022).

\bibitem{BD7} D.I. Borisov, D.A. Zezyulin, {\it On bifurcations of thresholds in essential spectrum under presence of spectral singularity} // Diff. Equat. {\bf 59}:2,  278--282  (2023).


\end{thebibliography}
\end{document}